\documentclass[12 pt]{amsart}
\usepackage{a4wide}
\newtheorem{thm}{Theorem}[section]

\newtheorem{prop}[thm]{Proposition}
\theoremstyle{definition}
\theoremstyle{definition}
\theoremstyle{definition}
\newtheorem{defn}[thm]{Definition}
\theoremstyle{remark}
\newtheorem{rem}[thm]{Remark}
\newtheorem{Example}[thm]{Example}
\numberwithin{equation}{section}

\newcommand{\T}{\mathbb{T}}
\newcommand{\R}{\mathbb R}
\newcommand{\Z}{\mathbb Z}

\def\1{\mathbb I}
\def\a{\alpha}
\def\b{\beta}
\def\d{\delta}

\def\g{\gamma}
\def\l{\lambda}

\def\o{\omega}

\def\eps{\varepsilon}

\def\ov{\overline}

\begin{document}
\title[]{Homogenization of  monotone   systems of Hamilton-Jacobi equations}
\author{Fabio Camilli, Olivier Ley \and Paola Loreti  }
\address{Dip. di Matematica Pura e Applicata, Univ. dell'Aquila, loc. Monteluco di Roio, 67040  l'Aquila,
Italy} \email{camilli@ing.univaq.it}
%%%
\address{Universit\'{e} Fran\c{c}ois-Rabelais, Tours.
Laboratoire de Math\'{e}matiques et Physique Th\'{e}\-ori\-que, CNRS UMR 6083,
F\'ed\'eration de Recherche Denis Poisson (FR 2964).
Facult\'{e} des Sciences et Techniques,
Parc de Grandmont, F-37200 TOURS, France} \email{ley@lmpt.univ-tours.fr}

\address{Dip. di Metodi e Modelli Matematici per le Scienze Applicate,
Facolt\`a  di Ingegneria,
 Universit\`a di Roma ``La Sapienza", via Scarpa 16, 00161 Roma, Italy}
\email{loreti@dmmm.uniroma1.it}

\begin{abstract}
In this paper we study  homogenization   for a class of
 monotone systems of first-order time-dependent periodic Hamilton-Jacobi equations.
We characterize  the Hamiltonians of the limit problem by appropriate cell problems. Hence  we
show the uniform convergence of the solution  of the oscillating systems to
the bounded
uniformly continuous solution of the
homogenized system.
\end{abstract}
\date{\today}
\maketitle
%%%%%%%%%%%%%%%%%%%%%%%%%%%%%%%%%%%%%%%%%%%%%%%%
%%%%%%%%%%%%%%%%%%%%%%%%%%%%%%%%%%%%%%%%%%%%%%%%
\section{Introduction}

In this paper we study the behavior as $\eps\to 0$ of the monotone system
of Hamilton-Jacobi equations
\begin{equation}\label{HJi}
\left\{
  \begin{array}{ll}
\frac{\partial u^\eps_i}{\partial t}+  H_i({x},\frac{x}{\eps},u^\eps, Du^\eps_i )=0
   & (x,t)\in  \R^N\times (0,T],   \\[5pt]
  u^\eps_{i}(x,0)=u_{0,i}(x)&x\in \R^N,\ i=1, \cdots M,
  \end{array}
\right.
\end{equation}
where   the Hamiltonians $H_i(x,y,r,p)$,
$i=1,\dots, M$, are periodic in $y,$ coercive in $p$ and
satisfy some uniform continuity properties, see \eqref{H1}.
The $u_{0,i}$'s are bounded uniformly continuous ($BUC$ in short). The
monotonicity condition, see  \eqref{A3},  we assume for
 the system is a standard assumption to obtain  a comparison principle for  \eqref{HJi} (see
\cite{el91}, \cite{ishii92}, \cite{ik91}, \cite{ik91b}).\par The main result
of the paper, see Theorem \ref{HT},  is the convergence of $u^\eps$, as $\eps\to 0$,  to a
$BUC$ function $u=(u_1,\dots,u_M)$ which
solves in viscosity sense the homogenized  system
\begin{eqnarray}\label{sys-homog}
\left\{
  \begin{array}{ll}
\frac{\partial u_i}{\partial t}+  \ov H_i(x,u, Du_i )=0
   & (x,t)\in  \R^N\times (0,T),\  \\[5pt]
  u_{i}(x,0)=u_{0,i}(x)&x\in \R^N\ i=1, \cdots M.
  \end{array}
\right.
\end{eqnarray}
The Hamiltonians $\ov H_i$ of the limit problem, the so-called
effective Hamiltonians, are characterized  by appropriate cell
problems. The comparison principle for \eqref{sys-homog} which provides
existence and uniqueness is not an imediate consequence of the comparison
principle for \eqref{HJi} since the regularity properties we could prove for
the effective Hamiltonians are weaker than those for the initial ones
(compare \eqref{H1} and \eqref{mod-hbarre}).\par
Homogenization of Hamilton-Jacobi equations in the framework of viscosity solution theory was firstly considered in the seminal paper by Lions, Papanicolau and Varadhan \cite{lpv86}.
The proof of the our homogenization result relies on an appropriate modification of the classical perturbed test function method. This
technique   was introduced  in the framework of the   viscosity solutions theory by Evans \cite{evans89} for the case of a   periodic equation.
Then it  has been  adapted to  many different homogenization problems,  see e.g. \cite{ab03}, \cite{ci01}, \cite{concordel96}, \cite{ls03}.
For a complete account of the homogenization theory in the periodic case we refer to \cite{ab03}.\par
Concerning the homogenization of  systems of Hamilton-Jacobi
equations we refer to \cite{evans89}, \cite{shimano06}. In these papers, homogenization of weakly coupled systems, i.e. systems with a linear coupling, was considered together with a penalization of the coupling term  of order $\eps^{-1}$.
 Because of the penalization,  the limit problem is a single Hamilton-Jacobi equation and
 all the components of the solution of the perturbed system converge  to the unique solution of this equation.\par
We consider the more general class of monotone systems, which in particular includes the weakly coupled ones. Moreover,   since we do not penalize  the coupling term, the homogenized problem is still a   system
of Hamilton-Jacobi equations and  the perturbed test function method has to be adapted to this situation.\par
In Section \ref{Ex}, we  discuss in more details  the homogenization of   the  weakly coupled systems. In particular   we show  that the homogenized
system is not necessarily weakly coupled but only monotone. For a particular 1-dimensional weakly coupled eikonal system, we  give an explicit
formula for the effective Hamiltonians.\par
The plan of the paper is the
following.\par In Section \ref{prelim} we describe our assumptions and definitions. In Section \ref{system}
we study the system   \eqref{HJi} for $\eps>0$. In Section
\ref{EH}, we define the effective Hamiltonians and we study their
properties. In Section \ref{Hom} we prove the homogenization
result. In Section \ref{Ex} we study some examples and in particular the
weakly coupled systems. Finally in the Appendix  we prove a
comparison theorem for \eqref{HJi}.\\[1mm]
\emph{Notation:}
We will use the following norm
\begin{eqnarray*}
    |f|_\infty=\mathop{\rm ess \,sup}_{x\in\R^N}|f(x)| \quad
\end{eqnarray*}
and $B_k(x,R)$ denotes the $k$-dimensional ball of center $x\in\R^N$ and radius $R>0.$

%%%%%%%%%%%%%%%%%%%%%%%%%%%%%%%%%%%%%%%%%%%%%%%%%%%%%%%%%%%%%%%%%%
\section{Assumptions and preliminary results}\label{prelim}
We consider the monotone system of Hamilton-Jacobi equations
\begin{equation}\label{HJS}
\left\{
\begin{array}{ll}
\displaystyle{\frac{\partial u^\eps_i}{\partial t}}+  H_i\left({x},\displaystyle{\frac{x}{\eps}}, u^\eps, Du^\eps_i \right)=0
   & (x,t)\in  \R^N\times (0,T),\ \ i=1, \cdots M, \\[10pt]
  u^\eps_{i}(x,0)=u_{0,i}(x)&x\in \R^N,
  \end{array}
\right.
\end{equation}
where $u^\eps=(u^\eps_1,\dots,u^\eps_M)$ and $u^\eps_i$ is a real valued function defined in $\R^N\times [0,T].$
We assume that the Hamiltonians $H_j:\R^N\times \R^N\times\R^M \times\R^N\to\R$,
$j=1,\dots,M$, are continuous and
satisfy the following assumptions:
\begin{eqnarray}
\label{H1}\left\{
\begin{array}{l}
\text{ i) $H_j(x,y,r,p)$ is $\Z^N$-periodic in $y$ for any $(x,r,p)$;}\\[6pt]
\text{ ii) $H_j(x,y,r,p)$ is coercive in $p$, i.e.}\\
\text{  \hspace*{0.5cm}
      $\displaystyle \lim_{|p|\to+\infty}H_j(x,y,r,p)=+\infty$  uniformly in $(x,y,r)$;}\\[8pt]
\text{ iii) For all $R>0,$
$H_j\in BUC(\R^N\times \R^N\times [-R,R]^M \times B_N(0,R))$;} \\[6pt]
\text{ iv) there exists a modulus of continuity
$\omega$ s.t}\\[3pt]
 \text{\hspace*{0.5cm}$|H_j(x_1, y_1, r,p)-H_j(x_2,y_2,r,p)|\le\o((1+|p|)(|x_1-x_2|+|y_1-y_2|)),$ }\\[3pt]
 \text{\hspace*{0.5cm}for every $x_1$, $x_2$,  $y_1$, $y_2$, $p\in\R^N$ and $r\in \R^M$.}
\end{array}\right.
\end{eqnarray}
Unless otherwise specified  all the periodic functions we consider have period $\T^N=[0,1]^N.$
We also assume the following monotonicity condition
\begin{equation}\label{A3}
    \begin{array}{c}
   \text{If $r$, $s\in\R^M$ and $\displaystyle r_j-s_j=\max_{k=1,\dots,M}\{r_k-s_k\}\ge 0$, then }\\[3pt]
   \text{ for all $x$, $y$, $p\in\R^N$, \ \
   $H_j(x,y,r,p)-H_j(x,y,s,p)\ge 0.$}
    \end{array}
\end{equation}
Concerning the initial datum we assume
\begin{eqnarray}
\text{$u_{0,j}$ is bounded uniformly continuous in $\R^N$ for $j=1,\dots,M$}.\label{H2}
\end{eqnarray}
%%%%%%%%%%%%%%
\begin{Example} \label{weak-exple} \ \\
1. Consider
\begin{eqnarray}\label{exemple-H}
H_j(x,y,r,p)= a_j(x,y)|p|+F_j(r).
\end{eqnarray}
where $a_j\in C(\R^N\times \R^N)$ and $F_j\in C(\R^M)$. If $a_j$ is
$\Z^N$-periodic in $y,$ then
\eqref{H1}i) holds. If there exists $\delta >0$ such that $a_j\geq \delta$ then
\eqref{H1}ii) is satisfied. If $a_j$ is bounded with respect to $x$ then
\eqref{H1}iii) holds (note that $a_j$ is bounded with respect to $y$ since
it is periodic). Finally, we have \eqref{H1}iv)
if, for instance, $a_j$ is Lipschitz continuous with respect to $(x,y).$
The assumption \eqref{A3}  is satisfied if $F_j$ is increasing in $r_j$, decreasing in $r_k$ for $k\neq j$.\\
2. A weakly coupled system is a system of the type
\begin{equation}\label{Ex1wk}
\frac{\partial u^\eps_i}{\partial t}+ H_i\left(x,\frac{x}{\eps},Du^\eps_i\right)
+\sum_{j=1}^M c_{ji}\left(x,\frac{x}{\eps} \right)u_j=0,\qquad i=1,\dots, M.
\end{equation}
 Some assumptions
on $c_{ij}$ to ensure \eqref{A3} are given in  Section \ref{Ex}.
Weakly coupled systems arise in optimal control theory of random evolution processes (see \cite{el91}). Moreover
they are associated to large deviation problems for small random perturbation of random evolution processes (see
\cite{cl07a}, \cite{ef90}).
We will study some specific case of weakly coupled systems in Section \ref{Ex}.
\end{Example}
%%%%%%%%%%%%%%%

%%%%%%
For a function $u:E\to \R^M$, we say that $u=(u_1,..,u_M)$ is
upper-semicontinuous (u.s.c in short), respectively lower-semicontinuous (l.s.c. in short), in $E$
if all the components $u_i$, $i=1,\dots, M$, are   u.s.c., respectively l.s.c., in $E$.
We define in the same way bounded uniformly continuous ($BUC$) and Lipschitz
continuous functions $u:E\to \R^M.$
 If $u=(u_1,\dots,u_M)$, $v=(v_1,\dots,v_M)$, are two functions
defined in a set $E$ we write  $u\le v$ in $E$ if $u_i\le v_i$ in $E$ for all $i\in \{1,\dots,M\}$.

We recall the definition of viscosity solution for the  system \eqref{HJS}.
%%%%%%%%%%%%%%%%
\begin{defn}
\item[i)] An u.s.c. function $u:  \R^N\times(0,T) \to\R^M$
is said a viscosity subsolution of \eqref{HJS} if $u_i(\cdot,0)\leq u_{0,i}$ in
$\R^N$ for all  $i\in \{1,\dots,M\}$ and if
whenever $\phi\in C^1$,
$i\in \{1,\dots,M\}$ and $u_i-\phi$ attains
a local maximum at $(x,t)$ with $t>0$, then
\[\frac{\partial \phi}{\partial t}(x,t)+H_i(x,\frac{x}{\eps},u(x,t),D\phi(x,t)) \le 0.\]
\item[ii)]  A l.s.c. $v:  \R^N\times(0,T) \to\R^M$ is said a viscosity  supersolution of
\eqref{HJS} if $v_i(\cdot,0)\geq u_{0,i}$ in
$\R^N$ for all  $i\in \{1,\dots,M\}$ and if whenever $\phi\in C^1$,
$i\in \{1,\dots,M\}$ and $v_i-\phi$ attains
a local minimum at $(x,t)$ with $t>0$, then
\[\frac{\partial \phi}{\partial t}(x,t)+H_i(x,\frac{x}{\eps},v(x,t),D\phi(x,t)) \ge 0.\]
\item[iii)] A continuous function $u$  is said a viscosity solution of  \eqref{HJS} if it is both
a viscosity sub- and supersolution of \eqref{HJS}.
\end{defn}
 %%%%%%%%%%%%%%
\section{The evolutive problem for $\eps>0$}\label{system}
In this section we study the system \eqref{HJS}  for $\eps>0$ fixed.
We first  prove a comparison theorem which applies to prove existence and uniqueness for \eqref{HJS}.
Without loss of generality, we can skip the $y$-dependence in the Hamiltonians
below and we prove a slightly more general result for $H_j=H_j(x,t,u,p)$ which depends also
on $t$ (and is continuous in $\R^N\times [0,T]\times\R^M \times\R^N$).
%%%%%%%%
\begin{prop}\label{PrComp}
Let $u$ be a bounded u.s.c. subsolution and $v$  be a bounded l.s.c. supersolution of
\begin{equation}\label{HJSa}
\left\{
  \begin{array}{ll}
\frac{\partial u_i}{\partial t}+  H_i({x}, t,u, Du_i )=0
   & (x,t)\in  \R^N\times (0,T), \\[5pt]
  u_{i}(x,0)=u_{0,i}(x)&x\in \R^N,\ \ i=1, \cdots M,
  \end{array}
\right.
\end{equation}
where $H_i$ satisfies \eqref{H1}-\eqref{A3} and $u_{0}$ satisfies \eqref{H2}.
Then $u\leq v$ in $\R^N\times [0,T]$
and there exists a unique continuous viscosity solution $u$
of \eqref{HJSa}.
\end{prop}
\begin{proof}
See  the Appendix \ref{Appendix}.
\end{proof}
%%%%%%%%
\begin{prop}\label{regBUC}
Under the assumptions of Proposition \ref{PrComp}, let $u$ be the
unique bounded continuous viscosity solution of \eqref{HJSa}.
Then $u\in BUC(\R^N\times [0,T]).$
\end{prop}
\begin{proof}
See  the Appendix \ref{Appendix}.
\end{proof}
%%%%%%%%%
In the following proposition, we prove
some a-priori bounds, independent of $\eps$, which are used in the  homogenization theorem.
%%%%%%%%%%%%%
\begin{prop}\label{PrUC}
Assume \eqref{H1} and \eqref{A3} and \eqref{H2}.
For any $\eps>0$ there exists a unique solution $u^\eps\in BUC(\R^N\times
[0,T])$ of \eqref{HJS}. Moreover
\begin{itemize}
\item[i)] If $u_0$ is bounded Lipschitz continuous, then
$u^\eps\in W^{1,\infty}(\R^N\times [0,T])$
and $|Du^\eps|_\infty$ can be bounded  independently of $\eps.$
\item[ii)] If $u_0$ is $BUC,$ then
\begin{eqnarray}
&& |u^\eps|_\infty \leq L(H_i, |u_0|_\infty, T)\label{PrUC1},\\
&&  |u^\eps(x,t)-u^\eps(y,s)|\leq \omega_{\rm sp}(|x-y|)+\omega_{\rm tm}(|t-s|)
\quad x,y\in\R^N,\,  t,s\in [0,T]\label{PrUC2}
\end{eqnarray}
and $ L(H_i, |u_0|_\infty, T), \omega_{\rm sp}, \omega_{\rm tm}$ are
independent of $\eps.$
\end{itemize}
\end{prop}
%%%%%%%%
\begin{proof}
For fixed $\eps >0,$
the existence and uniqueness of the solution $u^\eps\in BUC(\R^N\times [0,T])$
to \eqref{HJS} follows immediately by Propositions \ref{PrComp} and
\ref{regBUC}. Note that the $L^\infty$ bound for $u^\eps$ does not depend on
$\eps.$ Indeed, replacing $C$ in \eqref{def1C} by
\begin{eqnarray*}
C:=\sup\left\{ \left|H_j\left(x,\frac{y}{\eps},r,0\right)\right|: \, x\in\R^N,y\in\R^N,
|r|\le | u_0|_\infty,  \,1\leq j\leq M \right\}
\end{eqnarray*}
which is finite and independent of $\eps$ by periodicity of $H_j$ in $y,$
we obtain \eqref{linfini}  and therefore \eqref{PrUC1}.\par
We now prove \eqref{PrUC2}.
Let $u^\eps$ be a subsolution and $v^\eps$ be a supersolution of \eqref{HJS}
which are $BUC$ (the modulus of continuity of the
solution may a priori depend of $\eps$). Arguing as in Prop. \ref{PrComp}, from \eqref{blabla12}
we have, for all $j$ and $(x,t)\in\R^N\times [0,T],$
\begin{eqnarray*}
u_j^\eps(x,t)-v_j^\eps(x,t)
&\leq&
\eta t
+ (u_{\ov j}^\eps(\ov x,0)-v_{\ov j}^\eps (\ov y,0))^+ - \frac{|\ov x-\ov y|^2}{2\a}\\
&\leq&
\eta T +\mathop{\rm max}_{1\leq j\leq M}\mathop{\rm sup}_{\R^N}
(u_{j}^\eps(\cdot,0)-v_{j}^\eps(\cdot ,0))^+
+ (v_{\ov j}^\eps (\ov x,0)- v_{\ov j}^\eps (\ov y,0))^+ - \frac{|\ov x-\ov y|^2}{2\a}
\end{eqnarray*}
and
\begin{eqnarray*}
\mathop{\rm lim\,sup}_{\alpha\to 0} (v_{\ov j}^\eps (\ov x,0)- v_{\ov j}^\eps (\ov y,0))^+
- \frac{|\ov x-\ov y|^2}{2\a}
\leq 0
\end{eqnarray*}
since $v$ is uniformly continuous (see \eqref{unif-cont11} and \eqref{unif-cont22}).
Letting $\alpha\to 0$ and then $\eta \to 0,$ we obtain
\begin{eqnarray}\label{comp-max}
\mathop{\rm max}_{1\leq j\leq M}\mathop{\rm sup}_{\R^N\times [0,T]} u_j^\eps-v_j^\eps
\leq \mathop{\rm max}_{1\leq j\leq M}\mathop{\rm sup}_{\R^N}
\left(u_{j}^\eps(\cdot,0)-v_{j}^\eps(\cdot ,0)\right)^+.
\end{eqnarray}

We have to prove that the modulus of continuity of $u^\eps$ do not
depend on $\eps.$ We proceed by approximation showing first the
result for $u_0$ Lipschitz continuous. Replacing $C$ in \eqref{def1C} by
\begin{eqnarray*}
C:=\sup\left\{ \left|H_j\left(x,\frac{y}{\eps},r,p\right)\right|: \, x\in\R^N,\,y\in\R^N,\,
|r|\le | u_0|_\infty,\,  |p|\leq |Du_0|_\infty, \,1\leq j\leq M \right\},
\end{eqnarray*}
we prove as at the beginning of the proof of Proposition \ref{regBUC}
that, if $v^\pm (x,t)=( u_{0,1}(x) \pm Ct,\cdots , u_{0,M}(x) \pm Ct),$
then $v^+$ is a supersolution and $v^-$ is a subsolution of \eqref{HJS}.
By Proposition \ref{PrComp}, it follows
\begin{eqnarray}\label{comp999}
v^-\leq u \leq v^+ \quad {\rm in} \ \R^N\times [0,T].
\end{eqnarray}
Let $0\leq h\leq T$ and
note that, since the $H_i$'s are independent of $t,$
 $u^\eps (\cdot ,\cdot +h)$ is still a solution of \eqref{HJS}
with initial data $u^\eps (\cdot ,h).$
By \eqref{comp-max} and \eqref{comp999}, we get that, for all $j,$ $(x,t)\in \R^N\times [0,T],$
\begin{eqnarray*}
 u_j^\eps(x,t+h)-u_j^\eps(x,t)
\leq \mathop{\rm max}_{1\leq j\leq M}\mathop{\rm sup}_{\R^N}
\left(u_{j}^\eps(\cdot,h)-u_{0,j}\right)^+
\leq Ch
\end{eqnarray*}
and therefore $u_j^\eps$ is Lipschitz with respect to $t$ for every $x$ with
\begin{eqnarray*}
\left|\frac{\partial u_j^\eps}{\partial t}\right|_\infty \leq C.
\end{eqnarray*}
From   \eqref{HJS}, we obtain, in the viscosity sense
\begin{eqnarray*}
-C\leq H_j(x,\frac{x}{\eps},u^\eps, Du_j^\eps ) \leq C \qquad(x,t)\in\R^N\times [0,T].
\end{eqnarray*}
By the coercivity of $H_i$ (uniformly with respect to the other variables, see
\eqref{H1}ii)),
there exists $L_j>0$ such that, for all $p\in\R^N,$
\begin{eqnarray*}
|p|\geq L_j \ \ \Longrightarrow \ \ {\rm for \ all} \ \eps >0, x\in\R^N,  r\in\R^M, \
 H_j(x,\frac{x}{\eps},r, p)>C.
\end{eqnarray*}
It follows that $u_j^\eps$ is Lipschitz continuous in $x$ for every $t$ with
$|Du_j^\eps|_\infty \leq L_j$ (with $L_j$ independent of $\eps$).

Now if $u_0\in BUC(\R^N),$ then it is possible
to approach it by Lipschitz continuous functions: for all $\gamma>0,$ there
exists $u_0^\gamma$ such that $|u_0-u_0^\gamma|_\infty \leq \gamma.$
Let $u^\eps$ (respectively $u^{\eps,\gamma}$) be the unique $BUC$
(respectively Lipschitz continuous with constant $C_\gamma$) solution of \eqref{HJS} with initial data
$u_0$ (respectively $u_0^\gamma$). Note that $C_\gamma$ is independent of
$\eps.$ By \eqref{comp-max}, we obtain
\begin{eqnarray*}
|u^\eps - u^{\eps,\gamma}|_\infty \leq |u_0-u_0^\gamma|_\infty \leq \gamma.
\end{eqnarray*}
It follows that, for all $1\leq j\leq M,$ $x,y\in \R^N,$ $t,s\in [0,T],$
\begin{eqnarray}\label{est634}\hspace*{1cm}
|u_j^\eps (x,t) - u_j^{\eps}(y,s)| \leq
|u_j^{\eps,\gamma} (x,t) - u_j^{\eps,\gamma}(y,s)|+2\gamma
\leq C_\gamma (|x-y|+|t-s|)+2\gamma.
\end{eqnarray}
Since \eqref{est634} holds for all $\gamma >0$ and $C_\gamma$ is independent
of $\eps,$ we conclude that $u^\eps$ is $BUC$ with
a modulus independent of $\eps.$
\end{proof}.

%%%%%%%%%%%%%%%%%%%%%%%%%%%%%%%%%%%%%%%%%%%%%%%%%%%%%%%%%%%%%
\section{The cell problem}\label{EH}
In this section we prove the existence of the effective Hamiltonians, the
Hamiltonians for the limit system \eqref{sys-homog}.
Since at this level we work for a fixed index $i$, i.e.   there is no coupling,
we can follow the classical argument based on the ergodic approximation of the cell problem.
The only point is to prove that effective Hamiltonians we are going to define still verify some regularity and monotonicity properties
 so that the homogenized problem verifies a comparison principle.\\

\noindent\textbf{The Cell problem.}
For any $i=1,\dots,M$, given $(x,r,p)\in\R^N\times\R^M\times\R^N$, find $\l_i=\l_i(x,r,p)$ such that
the equation
\begin{equation}\label{CP}
    H_i(x,y,r,p+Dv(y))=\l_i\qquad y\in\T^N
\end{equation}
admits a   viscosity solution $v_i= v_{i,x,r,p}$.
%%%%%%%%%%%
\begin{prop}\label{PropCP}
Assume \eqref{H1}. For any $i=1,\dots,M$, there exists a unique $\l_i=\l_i(x,r,p)\in\R$
such that  the cell problem \eqref{CP} admits
a  periodic solution $v_i(y)=v_i(y;x,r,p)$ which is Lipschitz continuous.
More precisely, for all $R>0,$ there exists $L_R>0$ such that
 \begin{eqnarray*}
{\rm sup} \{ |D_y v_{i}(y;x,r,p)|_\infty \, : \, x\in\R^N, |r|+|p|\leq R \}\leq L_R.
\end{eqnarray*}
\end{prop}
%%%%%%%%%%%
\begin{proof}
We only give a sketch of the proof, see for instance \cite{concordel96} for details.
Fix $i\in\{1,\dots,M\},$ $R>0$ and $(x,r,p)\in \R^N\times\R^M\times\R^N$
such that $|r|+|p|\leq R.$
Consider the ergodic approximation of the cell problem
\begin{equation}\label{EA}
   \a w^\a_i(y)+H_i(x,y,r, p+Dw^\a_i(y))=0\qquad y\in \T^N.
\end{equation}
By \eqref{H1}, \eqref{EA} satisfies a comparison principle for any $\a>0$ and therefore it admits a unique
continuous viscosity solution $w^\a_i$ which is periodic. By \eqref{H1}iii),
 \begin{eqnarray*}
C_R:= {\rm sup} \{ |H_j(x,y,r,p)|\, :\,  x,y\in\R^N, |r|+|p|\leq R,
j=1,\cdots, M\}< +\infty
\end{eqnarray*}
and
$-C_R/\a$ is a subsolution  and  $C_R/\a$ is a supersolution of  \eqref{EA}.
It follows
 \begin{eqnarray*}
-C_R\leq \a w^\a_i\leq C_R.
\end{eqnarray*}
By the coercitivity of the Hamiltonian $H_i,$ there exists $L_R=L(R,H_i)$ such that
\begin{eqnarray*}
|p+Dw^\a_i|> L_R \ \ \Longrightarrow \ \ H_i(x,y,r,p+Dw^\a_i)>C_R.
\end{eqnarray*}
We then get the global gradient bounds for $w^\a_i$ independent of $\alpha:$
\begin{eqnarray*}
|Dw^\a_i|_\infty \leq L_R.
\end{eqnarray*}
It follows that, for a fixed $y_0\in\T^N$, there exists a sequence  $\a_n\to 0$ such that
\begin{align*}
&\lim_{n\to\infty} {\a_n} w^{\a_n}_i(y)=\l_i\qquad &\text{for any $y\in\T^N,$}\\
&\lim_{n\to\infty} w^{\a_n}_i(y)-w^{\a_n}_i(y_0)=v_i(y) &\text{uniformly in $\T^N$.}
\end{align*}
Moreover $v_i$ is Lipschitz continuous with constant $L_R$
and by standard stability result in viscosity solution theory $(\l_i, v_i)$ is a solution to \eqref{CP}.
Finally it is possible to prove that the number  $\l_i$ for which
\eqref{CP} admits a solution is univocally defined, while it is well known that
in general the viscosity solution of \eqref{CP} is not unique.
\end{proof}
%%%%%%%%%%%%%%%%

%%%%%%%%%%%
\begin{defn}
For any $i=1,\dots,M$, the effective Hamiltonian $\ov H_i(x,r,p)$ associated to the Hamiltonian $H_i$ is defined
by setting
\[\ov H_i(x,r,p)=\l_i\]
where $\l_i$ is given by Proposition \ref{PropCP}.
\end{defn}
%%%%%%%%%%%
We now deduce some  properties of the Effective Hamiltonians
%%%%%%%%%%%
\begin{prop} \label{prop-hbarre}
Assume \eqref{H1}.
For any $i=1,\dots,M$, the effective Hamiltonian $\ov H_i$ satisfies
\begin{itemize}
\item[i)] $\ov H_i$ is  continuous  in $(x,r,p)$ and, for all $R>0,$
there exists a modulus of continuity $\omega_R$ such that,
for all $x,x'\in \R^N,$ $r, r'\in \R^M,$ $p,p'\in \R^N$ with $|r|+|r'|+|p|+|p'|\leq R,$ we have
\begin{eqnarray}\label{mod-hbarre}
|\ov H_i (x,r,p)-\ov H_i (x',r',p')|\leq \omega_R(|x-x'|) + \omega_R(|r-r'|+|p-p'|).
\end{eqnarray}
\item[ii)] $\ov H_i$ is coercive in $p$ (uniformly with respect to $(x,r)$).
\item[iii)] If $H_i$ is convex in $p$, then $\ov H_i$ is convex in $p$.
\item[iv)]  If $H_i$ satisfies \eqref{A3}, then
 $\ov H_i$ satisfies \eqref{A3}.
\end{itemize}
\end{prop}
%%%%%%%%%%%
\begin{proof}
We first prove i). Let $R>0$ and $(x,r,p), (\a,s,q)\in\R^N\times\R^M\times \R^N$
such that $|r|+|s|+|p|+|q|\leq R.$
Let $v$, $w$ two periodic functions such that
\begin{eqnarray}
    &H_i(x,y,r,p+Dv(y))=\ov H_i(x,r,p)\qquad y\in \T^N, \label{eq111}\\
     &H_i(x+\a,y,r+s,p+q+Dw(y))=\ov H_i(x+\a,r+s,p+q)\qquad y\in \T^N. \label{eq222}
\end{eqnarray}
By periodicity of $v$ and $w,$ for any $\varepsilon >0,$ the supremum
\begin{eqnarray*}
\mathop{\rm sup}_{z,y\in\R^N} \{v(z)-w(y)-\frac{|z-y|^2}{\varepsilon^2}\}
\end{eqnarray*}
is achieved at some point $(\bar{z},\bar{y})$
(which depends on $x,\a,p,q,r,s,\varepsilon$).
Moreover it is easy to see that, since $v,w$ are bounded Lipschitz continuous,
we have
\begin{eqnarray}\label{sup-lip}
\frac{|\bar{z}-\bar{y}|}{\varepsilon^2}\leq |Dv|_\infty, |Dw|_\infty\leq L_R
 \end{eqnarray}
where $L_R$ is given by Proposition \ref{PropCP}.
Since  $v$ is a viscosity subsolution of \eqref{eq111} and $w$ is a supersolution
of \eqref{eq222}, we obtain
\begin{eqnarray*}
&H_i(x,\bar{z},r,p+\bar{p})\leq \ov H_i(x,r,p), \\
&H_i(x+\a,\bar{y},r+s,p+q+\bar{p})\geq \ov H_i(x+\a,r+s,p+q).
\end{eqnarray*}
It follows, using \eqref{H1}, that
\begin{eqnarray*}
&& \ov H_i(x+\a,r+s,p+q) - \ov H_i(x,r,p)\\
&\leq& H_i(x+\a,\bar{y},r+s,p+q+\bar{p})-H_i(x,\bar{z},r,p+\bar{p}) \\
&=& H_i(x+\a,\bar{y},r+s,p+q+\bar{p})-  H_i(x,\bar{z},r+s,p+q+\bar{p})\\
&& + H_i(x,\bar{z},r+s,p+q+\bar{p}) - H_i(x,\bar{z},r,p+\bar{p})\\
&\leq & \omega \big((1+R+L_R) (|\a|+|\bar{z}-\bar{y}|)\big)
+ \tilde \omega_{R} (|s|+|q|),
\end{eqnarray*}
where $\omega$ is given by \eqref{H1}iv) and $\tilde \omega_{R}$ is a modulus of
continuity of the continuous function $H_i$
on the subset $\R^N\times \T^N\times [-R,R]^M\times B_N(0,R+L_R)$ given by \eqref{H1}iii).
Sending $\varepsilon$ to 0 and setting $\omega_R(l)={\rm max}\{ \omega
((1+R+L_R)l), \tilde \omega_R (l)\},$
we get
\begin{eqnarray*}
 \ov H_i(x+\a,r+s,p+q) - \ov H_i(x,r,p)\leq \omega_R(|\a|)+ \omega_{R} (|s|+|q|),
\end{eqnarray*}
 which ends the proof of i).
The proof of \textit{ii)} and \textit{iii)} are standard, see \cite{concordel96}.

We now prove that
$\ov H_i$, $i=1,\dots, M$, satisfies the monotonicity condition \eqref{A3}. We assume by contradiction that
there exist  $r$, $s\in\R^M$ such that
$\displaystyle r_j-s_j=\mathop{\max}_{k=1,\dots,M}\{r_k-s_k\}\ge 0$ and
\begin{eqnarray*}
\ov H_j(x, r,p)<\ov H_j(x ,s,p)
\end{eqnarray*}
for some $x$, $p\in\R^N$. Let $v$, $w$  be two periodic  functions such that
\begin{align*}
    &H_j (x,y,r ,p+Dv)=\ov H_j(x,r,p)\qquad y\in \T^N,\\
     &H_j(x,y,s,p+Dw)=\ov H_j(x,s,p)\qquad y\in \T^N.
\end{align*}
Since $v$, $w$ are bounded, by adding a constant we can assume  w.l.o.g. that $v>w$ in $\T^N$.
By \eqref{A3}
\begin{eqnarray*}
H_j(x,y,r,p+Dv)=\ov H_j(x,r ,p)<\ov H_j(x,s,p)
\le  H_j(x,y,s,p+Dw)\le H_j(x,y,r,p+Dw)
\end{eqnarray*}
and   for $\a$ sufficiently small
\[
\a v+H_j(x,y,r,p+Dv)>\a w+ H_j(x,y,r,p+Dw)\qquad y\in\T^N.
\]
This last inequality gives a  contradiction by the comparison principle for
\eqref{EA}.
\end{proof}
%%%%%%%%%%%

%%%%%%%%%%%%%%%%%%%%%%%%%%%%%%%%%%%%%%%%%%%%%%%%%%%%%%%%%%%%%%%%
\section{The Homogenization theorem}\label{Hom}
In this section we prove the Homogenization theorem for the problem \eqref{HJi}.

%%%%%%%%%
\begin{prop} \label{comp-barre}
Assume \eqref{H1}, \eqref{A3} and \eqref{H2}.
Then there exists a unique solution $u\in BUC(\R^N\times[0,T])$
of
\begin{equation}\label{HJ0}
\left\{
  \begin{array}{ll}
  \displaystyle{\frac{\partial u_i }{\partial t}}+  \ov H_i(x,u , Du_i  )=0& (x,t)\in \R^N\times (0,\infty),
  \\[10pt]
  u_i (x,0)=u_{0,i} (x)&x\in \R^N, \quad i=1,\dots,M.
  \end{array}
\right.
\end{equation}
\end{prop}
%%%%%%%
\begin{proof}
The difficulty here is that the comparison principle for the limit system
\eqref{comp-barre} is not a straightforward consequence of Proposition
\ref{prop-hbarre}. Indeed, the regularity of  the Hamiltonians $\ov H_i$ is weaker
than \eqref{H1} (in particular compare \eqref{H1}iv) and \eqref{mod-hbarre}).
To prove the comparison principle we first prove comparison in the case
where either the subsolution or the supersolution is bounded Lipschitz continuous
and then we proceed by approximation.\par
Suppose that $u_0$ is bounded Lipschitz continuous, that $u$ is a bounded subsolution
and $v$ a bounded supersolution  of \eqref{HJ0} and that $u,$ for instance,
is Lipschitz continuous (with constant $L$).
Arguing as in Proposition \ref{PrComp} and looking carefully at the proof of
Proposition \ref{prop-hbarre} (see in particular \eqref{sup-lip}),
it follows that the second estimate in \eqref{mie21} could be replaced by
\begin{eqnarray*}
|\ov p|\leq L,
\end{eqnarray*}
and therefore, setting $R=|v|_\infty + 2M+L,$ from \eqref{mod-hbarre},
\eqref{T111} becomes
\begin{eqnarray*}
\mathcal{T}_1\leq \omega_R (L\alpha).
\end{eqnarray*}
The term $\mathcal{T}_2$ does not exist since $\ov H_{\ov j}$ does not depend
on $t.$ We deal with $\mathcal{T}_3$ using again \eqref{mod-hbarre} and
$\mathcal{T}_4\leq 0$ as in the proof of  Proposition \ref{prop-hbarre}.
The rest of the proof is the same (even easier since $u(\cdot ,0)$ is
Lipschitz continuous). It follows that we have comparison between Lispchitz
continuous sub and supersolutions. In particular, by Perron's method,
for any Lipschitz continuous $u_0,$ there exists a unique Lipschitz
continuous solution $u$ of \eqref{HJ0}. Moreover, repeating the beginning
of the proof of Proposition \ref{PrUC}, we obtain that \eqref{comp-max}
holds, i.e.
\begin{eqnarray}\label{val-ab}
\mathop{\rm max}_{1\leq j\leq M}\mathop{\rm sup}_{\R^N\times [0,T]} u_j-v_j
\leq \mathop{\rm max}_{1\leq j\leq M}\mathop{\rm sup}_{\R^N}
\left(u_{j}(\cdot,0)-v_{j}(\cdot ,0)\right)^+
\end{eqnarray}
if $u$ is a subsolution and $v$ a supersolution of \eqref{HJ0} and
either $u$ or $v$ is Lipschitz continuous.

Now, consider the case when $u_0$ is $BUC.$
Let $u$ (respectively $v$) be a $BUC$
subsolution (respectively supersolution) of \eqref{HJ0}.
For all $\gamma>0,$ there exists a Lipschitz continuous function $u_0^\gamma$
such that
\begin{eqnarray}\label{approx}
u_0^\gamma \leq u_0 \leq u_0^\gamma+\gamma \quad {\rm in} \ \R^N.
\end{eqnarray}
Let $u^\gamma$ (respectively $v^\gamma$) be the Lipschitz continuous solution to \eqref{HJ0} with
initial data $u_0^\gamma$ (respectively $u_0^\gamma+\gamma$).
By comparison in the Lipschitz case, $u^\gamma\leq v^\gamma.$
From \eqref{val-ab} and \eqref{approx}, it follows
\begin{eqnarray*}
u\leq u^\gamma +\gamma \leq v^\gamma +\gamma\leq v+2\gamma \quad {\rm in} \
\R^N\times [0,T].
\end{eqnarray*}
Since the previous inequality is true for all $\gamma>0,$
we obtain the desired comparison $u\leq v.$ Note that we obtain the existence
and the uniqueness of a $BUC$ solution as a
byproduct of this latter proof.
\end{proof}

%%%%%
\begin{thm}
\label{HT} Assume \eqref{H1},  \eqref{A3} and \eqref{H2}.
 The viscosity solution $u^\eps$ of \eqref{HJi} converges locally  uniformly
on $\R^N\times [0,T]$ to the viscosity solution $u\in BUC(\R^N\times [0,T])$ of \eqref{HJ0}.
\end{thm}
%%%%%

%%%%%%%%%%%%
\begin{proof}
By Proposition \ref{PrUC}  there exists a
continuous solution $u^\eps$ of \eqref{HJi}
which is bounded independently of $\eps.$ It follows that we can define the half-relaxed
limits
\begin{eqnarray*}
\ov u (x,t) = \mathop{\rm lim\,sup}_{\eps\to 0, (x_\eps, t_\eps)\to(x,t)} u^\eps(x_\eps, t_\eps)
\ \ \ {\rm and} \ \ \
\underline u (x,t) = \mathop{\rm lim\,inf}_{\eps\to 0, (x_\eps, t_\eps)\to(x,t)} u^\eps(x_\eps, t_\eps).
\end{eqnarray*}
Let us mention at this step that we could use Ascoli's theorem in view of the
equicontinuity property  of Proposition \ref{PrUC} ii) to obtain a limit for
$u^\eps$ along a subsequence. We choose to use the
half-relaxed limits since it is not much more complicated and it does not require
uniform moduli of continuity for the $u^\eps$'s.

We first show that $\ov u$ is a viscosity subsolution of the system \eqref{HJ0}.
We assume that there exist  $j\in\{1,\dots,M\}$ and $\phi\in C^1$ such that
$\ov{u}_j-\phi$ has a strict maximum point at some $(\ov x,\ov t)$ with
$\ov t>0$ and
$\ov u_j(\ov x,\ov t)=\phi(\ov x,\ov t).$
We assume w.l.o.g. that $j=1$. By Proposition \ref{PropCP}, there exists a
corrector $v$ for $(\ov x, \ov u(\ov x,\ov t), D\phi(\ov x,\ov t))$, i.e. a viscosity solution of
\begin{equation}\label{T2-new}
    H_1(\ov x,y,\ov u(\ov x,\ov t),D\phi(\ov x,\ov t)+Dv(y))=\ov H_1(\ov x,\ov u(\ov x,\ov t),D\phi(\ov x,\ov t))\qquad y\in\T^N.
\end{equation}
Define the ``perturbed test-function''
\begin{equation}\label{T3-new}
   \phi^{\eps,\a}(x,y,t)=\phi(x,t)+\eps v\left(\frac{y}{\eps}\right)+\frac{|x-y|^2}{\alpha^2}.
\end{equation}
By classical results on viscosity solutions (see \cite[Lemma 4.3]{barles94} or
\cite{bcd97}), we have, since $\ov{u}_1-\phi$ has a strict maximum point at $(\ov x,\ov t),$
up to extract subsequences, there exist
$(x_{\eps,\alpha}, y_{\eps,\alpha}, t_{\eps,\alpha})\in\R^N\times\R^N\times (0,T]$
and $(x_{\eps}, t_{\eps})\in\R^N\times (0,T]$ such that
$(x_{\eps,\alpha}, y_{\eps,\alpha}, t_{\eps,\alpha})$ is a local maximum
of ${u}^\eps_1(x,t)-\phi^{\eps,\a}(x,y,t)$ and
\begin{eqnarray*}
&& (x_{\eps,\alpha}, y_{\eps,\alpha}, t_{\eps,\alpha}) \to (x_{\eps}, x_{\eps},t_{\eps})
\ \ \ {\rm as} \ \alpha\to 0, \\
&& (x_{\eps}, t_{\eps}) \to (\ov x,\ov t)\ \ \ {\rm as} \ \eps\to 0, \\
&& \mathop{\rm lim}_{\eps\to 0}\,  \mathop{\rm lim}_{\alpha\to 0}
{u}^\eps_1(x_{\eps,\alpha},t_{\eps,\alpha})= u_1 (\ov x, \ov t).
\end{eqnarray*}
%%%%
Since $u^\eps_1(x,t)- \phi^{\eps,\alpha}(x,y_{\eps,\alpha},t)$ has a maximum point at
$(x_{\eps,\alpha},t_{\eps,\alpha})$ and $u^\eps$ is a subsolution of \eqref{HJi},
setting $\displaystyle p_{\eps,\a}=2\frac{x_{\eps,\alpha}-y_{\eps,\alpha}}{\alpha^2},$
we get
\begin{equation}\label{visc66bis}
\phi_t(x_{\eps,\alpha},t_{\eps,\alpha})
+H_1(x_{\eps,\alpha},\frac{x_{\eps,\alpha}}{\eps},{u}^\eps(x_{\eps,\alpha},t_{\eps,\alpha}),
D\phi(x_{\eps,\alpha},t_{\eps,\alpha})+ p_{\eps,\a})
\leq 0.
\end{equation}
Since  $v$ is a
supersolution of \eqref{T2-new} and $y\mapsto v(y)-\psi^{\eps,\alpha}(\eps y)$
 has a minimum point at $y_{\eps,\a}/\eps$ with
\begin{eqnarray*}
\psi^{\eps,\alpha}(y)=-\frac{1}{\eps}\left(\frac{|x_{\eps,\a}-y|^2}{\alpha^2}+
\phi(x_{\eps,\a}, t_{\eps,\a})-u^\eps_1(x_{\eps,\a}, t_{\eps,\a})\right),
\end{eqnarray*}
we get
\begin{equation}
H_1(\ov x,\frac{y_{\eps,\alpha}}{\eps},\ov u(\ov x,\ov t),D\phi(\ov x,\ov t)
+p_{\eps,\a})
\geq  \ov H_1(\ov x,\ov u(\ov x,\ov t),D\phi(\ov x,\ov t)). \label{visc66}
\end{equation}
Note that the corrector $v$ is Lipschitz continuous by coercivity of $H_1$
(see Proposition \ref{PropCP}). Therefore,
from \eqref{sup-lip}, we have
\begin{eqnarray}\label{bornitude-p}
 |p_{\eps,\a}|\leq |Dv|_\infty \leq L_R \ \ \ {\rm with} \ R=|D\phi(\ov x,\ov t)|+|\ov u|_\infty.
\end{eqnarray}

By \eqref{visc66bis} and \eqref{visc66}, we have
\begin{eqnarray} \label{eg876}
&&\phi_t (\ov x,\ov t)+\ov H_1(\ov x, \ov u(\ov x,\ov t),D\phi(\ov x,\ov t))\\
&\leq & \phi_t (\ov x,\ov t)-\phi_t (x_{\eps,\alpha},t_{\eps,\alpha})\nonumber\\
& & + H_1(\ov x,\frac{y_{\eps,\alpha}}{\eps},\ov u(\ov x,\ov t),
D\phi(\ov x,\ov t)+p_{\eps,\a})\nonumber\\
&& - H_1(x_{\eps,\alpha},\frac{x_{\eps,\alpha}}{\eps}, {u}^\eps(x_{\eps,\alpha},t_{\eps,\alpha}),
D\phi(x_{\eps,\alpha},t_{\eps,\alpha})+ p_{\eps,\a}).\nonumber
\end{eqnarray}

Since $(x_{\eps,\alpha},t_{\eps,\alpha})\to (\ov x,\ov t)$ and $\phi$ is smooth we get
\begin{equation}\label{esT1}
\mathop{\rm lim}_{\eps\to 0}\,  \mathop{\rm lim}_{\alpha\to 0}
\phi_t (\ov x,\ov t)-\phi_t (x_{\eps,\alpha},t_{\eps,\alpha})=0.
\end{equation}
To estimate the second term of the right-hand side we set
\begin{eqnarray*}
\mathcal{T}_1&=&H_1(\ov x,\frac{y_{\eps,\alpha}}{\eps},\ov u(\ov x,\ov t),D\phi(\ov x,\ov t)+ p_{\eps,\a})
- H_1(x_{\eps,\alpha},\frac{x_{\eps,\alpha}}{\eps},\ov u(\ov x,\ov t),D\phi(\ov x,\ov t)+ p_{\eps,\a}), \\
\mathcal{T}_2&=& H_1(x_{\eps,\alpha},\frac{x_{\eps,\alpha}}{\eps},\ov u(\ov x,\ov t),D\phi(\ov x,\ov t)+ p_{\eps,\a})
-H_1(x_{\eps,\alpha},\frac{x_{\eps,\alpha}}{\eps}, {u}^\eps(x_{\eps,\alpha},t_{\eps,\alpha}),
D\phi(\ov x ,\ov t)+ p_{\eps,\a}),\\
\mathcal{T}_3&=&H_1(x_{\eps,\alpha},\frac{x_{\eps,\alpha}}{\eps}, {u}^\eps(x_{\eps,\alpha},t_{\eps,\alpha}),
D\phi(\ov x ,\ov t)+ p_{\eps,\a})\\
&& -H_1(x_{\eps,\alpha},\frac{x_{\eps,\alpha}}{\eps}, {u}^\eps(x_{\eps,\alpha},t_{\eps,\alpha}),
D\phi(x_{\eps,\alpha},t_{\eps,\alpha})+ p_{\eps,\a}).
 \end{eqnarray*}
From \eqref{H1}iv) and \eqref{bornitude-p},
\begin{eqnarray*}
\mathcal{T}_1 &\leq& \omega \big( (1+|D\phi(\ov x,\ov t)+ p_{\eps,\a}|)
(|\ov x -x_{\eps,\alpha}|+\frac{|x_{\eps,\alpha}-y_{\eps,\alpha}|}{\eps})\big)\\
&\leq& \omega \big( (1+R+L_R)(|\ov x -x_{\eps,\alpha}|+\frac{L_R\alpha^2}{\eps})\big)
\end{eqnarray*}
and therefore
\begin{eqnarray*}
 \mathop{\rm lim}_{\eps\to 0}\,  \mathop{\rm lim}_{\alpha\to 0}\mathcal{T}_1 =0.
\end{eqnarray*}

To deal with $\mathcal{T}_2,$ we use the monotonicity assumption \eqref{A3}.
Let $\delta >0.$
At first, up to extract some subsequences, by definition of $\ov u,$
we can assume that for $\alpha, \eps$
small enough with $\alpha <\!\! < \eps,$ we have
\begin{eqnarray*}
u_j^\eps(x_{\eps,\alpha},t_{\eps,\alpha})- \ov u_j(\ov x,\ov t)\leq \frac{\delta}{2}
\ \ \ {\rm for} \ 2\leq j\leq M.
\end{eqnarray*}
It follows that
\begin{eqnarray*}
{\rm max}\big\{  \ov u_1(\ov x,\ov t)+\delta- \ov u_1(\ov x,\ov t),
u_2^\eps(x_{\eps,\alpha},t_{\eps,\alpha}) -\ov u_2(\ov x,\ov t),\cdots,
u_M^\eps(x_{\eps,\alpha},t_{\eps,\alpha}) -\ov u_M(\ov x,\ov t)
\big\}=\delta
\end{eqnarray*}
is achieved for the first component. Set $r_\d=(\ov u_1(\ov x,\ov t)+\delta,
u_2^\eps(x_{\eps,\alpha},t_{\eps,\alpha}),\cdots,
u_M^\eps(x_{\eps,\alpha},t_{\eps,\alpha}))$, then by \eqref{A3}
\begin{eqnarray*}
H_1(x_{\eps,\alpha},\frac{x_{\eps,\alpha}}{\eps},\ov u(\ov x,\ov t),
D\phi(\ov x,\ov t)+ p_{\eps,\a})-H_1(x_{\eps,\alpha},\frac{x_{\eps,\alpha}}{\eps}, r_\d,
D\phi(\ov x ,\ov t)+ p_{\eps,\a})
\leq 0.
\end{eqnarray*}
Then
\begin{eqnarray*}
\mathcal{T}_2\leq
H_1(x_{\eps,\alpha},\frac{x_{\eps,\alpha}}{\eps},r_\d,D\phi(\ov x,\ov t)+p_{\eps,\a})
-H_1(x_{\eps,\alpha},\frac{x_{\eps,\alpha}}{\eps}, {u}^\eps(x_{\eps,\alpha},t_{\eps,\alpha}),
D\phi(\ov x ,\ov t)+ p_{\eps,\a}):= \mathcal{T}_4.
\end{eqnarray*}

To prove that
\begin{eqnarray} \label{limlim}
 \mathop{\rm lim}_{\eps\to 0}\,  \mathop{\rm lim}_{\alpha\to 0}\mathcal{T}_3 = 0
\ \ \ {\rm and} \ \ \
\mathop{\rm lim}_{\delta\to 0}\,\mathop{\rm lim}_{\eps\to 0}\,  \mathop{\rm lim}_{\alpha\to 0}
 \mathcal{T}_4 = 0,
\end{eqnarray}
we use the uniform continuity of $H_1$ on compact subsets.
We have
\begin{eqnarray*}
\mathop{\rm lim}_{\eps\to 0}\,  \mathop{\rm lim}_{\alpha\to 0}\,
(x_{\eps,\alpha}, {u}^\eps(x_{\eps,\alpha},t_{\eps,\alpha}),D\phi(x_{\eps,\alpha},t_{\eps,\alpha}))
= (\ov x, \ov u(\ov x,\ov t), D\phi(\ov x ,\ov t))
\end{eqnarray*}
and
\begin{eqnarray*}
\mathop{\rm lim}_{\delta\to 0}\,\mathop{\rm lim}_{\eps\to 0}\,  \mathop{\rm lim}_{\alpha\to 0}\,
r_\delta = \ov u(\ov x,\ov t).
\end{eqnarray*}
Since $x_{\eps,\a}, y_{\eps,\a}\to \ov x$ we have
that $x_{\eps,\a}, y_{\eps,\a}$ stay in some ball $\ov B(\ov x,\ov R)$.
Hence choosing
$K=\ov B(\ov x,\ov R)\times \T^N\times
[-|\ov u|_\infty-1, |\ov u|_\infty+1]^M\times \ov B(0, |D\phi(\ov x,\ov t)|+L_R+1),$
by uniform continuity of $H_1$ on $K,$ \eqref{limlim} holds. Note that
the periodicity of $H_1$ allows
to deal with ${x_{\eps,\alpha}}/{\eps}$ which is not bounded.

Finally, sending $\alpha\to 0$ at first, then $\eps\to 0$ and finally
$\delta\to 0,$ we conclude that the right-hand side of \eqref{eg876}
is nonpositive which proves that $\ov u$ is a subsolution of
\eqref{HJ0}.

We prove that $\underline u$ is a viscosity supersolution of \eqref{HJ0} in a
similar way. From Proposition \ref{PrComp}, we then obtain that $\ov u\leq
\underline u$ in $\R^N\times [0,T].$ It follows that $\ov u=\underline u:=u$
where $u$ is the (local) uniform limit of the $u^\eps$'s.
\end{proof}

%%%%%%%%%%
\begin{rem}\label{corr-non-lip} \ \\
1. As mentioned above, the coercivity of the Hamiltonians plays a crucial role: it ensures
the Lipschitz continuity of the correctors which allows us to deal with
$\mathcal{T}_3$ and $\mathcal{T}_4$ with weak regularity assumptions with respect to $(r,p)$
in \eqref{H1}.
When the Hamiltonians are not coercive anymore (and therefore the corrector
is not necessarily Lipschitz continuous), the proof is more delicate.
A way to solve this problem is to use the ideas of
Barles \cite{barles07} and his ''$F^k$-trick'' (see \cite[Lemma 2.1 and
Theorem  2.1]{barles07}). \\
2. In the Lipschitz case (when $u_0$ is Lipschitz continuous), the above proof
can be done in a simpler way using the uniform Lipschitz estimates on $u^\eps$ given
by Proposition \ref{PrUC}.
\end{rem}
%%%%%%

%%%%%%%%%%%%%%%%%%%%%%%%%%%%%%%%%%
%%%%%%%%%%%%%%%%%%%%%%%%%%%%%%%%%%
%%%%%%%%%%%%%%%%%%%%%%%%
\section{Example}\label{Ex}
We first describes a class of   systems \eqref{HJS} which satisfy \eqref{A3}. We assume that the Hamiltonians $H_j$ satisfy the following
assumption
(see \cite{el91})
\begin{equation}\label{A1}
\begin{array}{c}
\text{There  exists $c_{ij}\in\R,$ $1\leq i,j\leq M,$ s.t. $\sum_{j=1}^M c_{ji}\ge  0$ and}\\[4pt]
\text{for any $(x,y,r,p)\in\R^N\times\R^N\times\R^M\times \R^N$, $\d>0,$}\\[4pt]
 c_{ji} \d\le H_i(x,y,r+\d e_j,p)- H_i(x,y,r,p)\le 0\qquad \text{if $j\neq i$},\\[4pt]
  c_{ii} \d \le H_i(x,y,r+\d e_i,p)- H_i(x,y,r,p)
\end{array}
\end{equation}
where $(e_1,\cdots ,e_n)$ is the canonical basis of $\R^M.$
Note that necessarily $c_{ji}\leq 0$
for $i\not= j$ and $c_{ii}\geq  0$. \par

In the next proposition we prove that the  assumption \eqref{A1}  implies the      monotonicity
 condition \eqref{A3}
%%%%%%%%%%
\begin{prop}
Condition \eqref{A1} implies  \eqref{A3}.
\end{prop}
%%%%%%%%%%
\begin{proof}
Assume that $r_j-s_j=\max_{k=1,\dots,M}\{r_k-s_k\}\ge 0$. For simplicity, we
drop the dependence in $(x,y,p)$ in $H(x,y,r,p)$ in the proof of the proposition since
these variables do not play any role here.
We have
\begin{eqnarray*}
&& H_j(r)-H_j(s) \\
&=& H_j(r_1, r_2, \cdots , r_j, \cdots , r_M) -   H_j(s_1, s_2, \cdots , s_j, \cdots , s_M) \\
&=& H_j(r_1, r_2, \cdots , r_j, \cdots , r_M) -  H_j(s_1, r_2, \cdots , r_j, \cdots , r_M) \\
&& + H_j(s_1, r_2, \cdots , r_j, \cdots , r_M) -  H_j(s_1, s_2, r_3, \cdots ,r_j, \cdots , r_M) \\
&& + \cdots \\
&& + H_j (s_1,\cdots , s_{j-1}, r_j,\cdots r_M)- H_j (s_1,\cdots , s_{j}, r_{j+1},\cdots r_M) \\
&& + \cdots \\
&& + H_j (s_1,\cdots , s_{M-1}, r_M)- H_j (s_1, \cdots s_M).
\end{eqnarray*}
If $k\not= j,$
\begin{eqnarray*}
&& H_j(s_1, \cdots s_{k-1}, r_k,r_{k+1} \cdots , r_M) -  H_j(s_1, \cdots ,s_{k-1},s_k , r_{k+1}, \cdots , r_M)\\[5pt]
&\geq&
\left\{
\begin{array}{cc}
0 & {\rm if} \ r_k-s_k <0\\
c_{kj}(r_k-s_k) & {\rm if} \ r_k-s_k \geq 0
\end{array}\right.
\quad \geq  c_{kj}(r_j-s_j)
\end{eqnarray*}
since $c_{kj}\leq 0$ and $r_j-s_j=\max_{k=1,\dots,M}\{r_k-s_k\}\ge 0$. Moreover,
\begin{eqnarray*}
 H_j(s_1, \cdots s_{j-1}, r_j,r_{j+1} \cdots , r_M) -  H_j(s_1, \cdots s_{j-1},s_j , r_{j+1}, \cdots , r_M)
\geq c_{jj}(r_j-s_j).
\end{eqnarray*}
It follows
\begin{eqnarray*}
H_j(r)-H_j(s) \geq \sum_{k=1}^{M} c_{kj}(r_j-s_j)\geq 0
\end{eqnarray*}
as desired.
\end{proof}
%%%%%%%%%%
\begin{rem}
Property \eqref{A3} is not equivalent to \eqref{A1}. More
precisely, if \eqref{A3} holds, the existence of $c_{ji}$ for $j\not= i$ is not always true
(the others assertions hold).
Indeed, for
$M=2,$ consider for instance $H_1 (r_1,r_2)= {\rm e}^{r_1-r_2}+2r_1 -r_2$
(and define $H_2$ symmetrically). Then
\begin{eqnarray*}
H_1( r_1+\delta , r_2+\mu)- H_1(r_1,r_2)=  {\rm e}^{r_1-r_2}({\rm  e}^{\delta-\mu} -1)+2\delta -\mu
\geq \delta
\end{eqnarray*}
when $\mu\leq \delta.$
This ensures \eqref{A3} with $\lambda_0=1.$ Nevertheless,
$H_1( r_1 , r_2+\mu)- H_1(r_1,r_2) \sim -(1+{\rm e}^{r_1-r_2}) \mu$ for small
$\mu$
and $1+{\rm e}^{r_1-r_2}$ is not bounded.
\end{rem}
%%%%%%%%%%%%%
A particular case of monotone systems are the    weakly coupled  systems \eqref{Ex1wk}.
For \eqref{Ex1wk}, assumption  \eqref{A1} is satisfied if
\begin{eqnarray}\label{cond-sur-c}
c_{ii}(x,y)\ge 0, \ \ c_{ji}(x,y)\le 0 \ \text{ for $j\neq i$}\ \ \text{ and}
\ \ \sum_{j=1}^M
c_{ji}(x,y)\geq 0
\end{eqnarray}
for any $x,y\in\R^N$, $i,j\in\{1,\dots,M\}$.
Let us consider a specific example of
weakly coupled system
for which it is possible to have an explicit
formula for the effective Hamiltonians. Consider the system
\[
\frac{\partial u_i}{\partial t}+ \left|\frac{\partial u_i}{\partial x}\right|
+\sum_{j=1}^M c_{ji}\left(\frac{x}{\eps}\right)u_j=0\qquad (x,t)\in \R\times[0,\infty).
\]
where  the $c_{ji}$'s satisfy \eqref{cond-sur-c}. The associated cell problems are
\begin{equation}\label{Ex1CP}
|p+v'(y)|+\sum_{j=1}^M c_{ji}(y)r_j=\l\qquad y\in [0,1], \lambda\in\R,
\end{equation}
for $i=1,\dots,M$.  We rewrite \eqref{Ex1CP} as
\begin{equation}\label{Ex1CP1}
|p+v'(y)|=\l+f(y)\qquad y\in [0,1]
\end{equation}
where $\displaystyle f(y)=-\sum_{j=1}^Mc_{ji}(y)r_j$.
The effective Hamiltonian for \eqref{Ex1CP1} is given by (see \cite{concordel96})
\begin{equation}\label{Ex1EH}
    \ov H(p)=\max\{-\min_{[0,1]}f,\, |p|-\int_0^1 f(y)dy\}
\end{equation}
and therefore we get the effective Hamiltonian for \eqref{Ex1CP}
\begin{equation}\label{Ex1F}
\ov H_i(r,p)=\max\left\{\max_{[0,1]}\sum_{j=1}^M c_{ji}(y)r_j,\, |p|+\sum_{j=1}^M r_j\int_0^1 c_{ji}(y)dy\right\}
\end{equation}

A natural question is that if the problem \eqref{HJ0} which arises in the homogenized
limit of the weakly coupled system \eqref{Ex1wk} is still of weakly coupled type. Whereas the answer is
positive if the coefficients $c_{ij}$ are independent of $y$, in general it is
not necessarily true.

%%%%%%%%%
\begin{prop} Assume that \eqref{H1} holds and
that the coefficients $c_{ij}$ in \eqref{Ex1wk} are constant and
satisfy \eqref{cond-sur-c}. Then
\begin{equation}\label{Ex1wk1}
   \ov H_i(r,p)=\ov H_i(p)+  \sum_{j=1}^M c_{ji}\,r_j
\end{equation}
where $\ov H_i(p)$ is the effective Hamiltonian of $H_i(x,p)$, i.e. the unique $\l\in \R$ for which the equation
\begin{equation}\label{Ex1wk2}
    H_i(y,p+Dv(y))=\ov H_i(p)\qquad y\in\T^N
\end{equation}
admits a viscosity solution.
\end{prop}
%%%%%%%
\begin{proof}
By definition, there exists a  viscosity solution to
\[H_i(y,p+Dv(y))+\sum_{j=1}^M c_{ji}r_j=\ov H_i(r,p),\qquad y\in\T^N\]
or equivalently to
\[H_i(y,p+Dv(y))=\ov H_i(r,p)- \sum_{j=1}^M c_{ji}r_j\qquad y\in\T^N.\]
By \eqref{Ex1wk2} and the uniqueness of the effective Hamiltonian, the
constant in the right hand side of  the previous equation \ is given by $\ov H_i(p)$,
hence the formula \eqref{Ex1wk1}.
\end{proof}
%%%%%%%%%

The following example
shows that if the coupling coefficients $c_{ij}$ are not constants, the limit system is not
 necessarily weakly coupled. Consider the 1-dimensional case \eqref{Ex1CP}.  Take $i=1$ and $r=(r_1,0,\dots,0)$, then
 \[\ov H_1(r,p)=\max\left\{\max_{[0,1]} c_{11}(y)r_1,\, |p|+  r_1\int_0^1 c_{11}(y)dy\right\}.\]
If $\max_{[0,1]} c_{11} =\a$, $\min_{[0,1]} c_{11}=\b$ and $\int_0^1 c_{11}(y)dy=\g$, with $\a>\g>\b\ge 0$, then
for $p\neq 0$ fixed
\[ \ov H_1(r,p)=\left\{
\begin{array}{ll}
\b r_1& {\rm if} \ r_1\le  |p|/(\b-\g),\\
\g r_1+|p|& {\rm if} \ |p|/(\b-\g)\le r_1\le |p|/(\a-\g),\\
\a r_1&  {\rm if} \ r_1\ge |p|/(\a-\g).
\end{array}
\right.\]
Then $H_1(r,p)$ is not a linear function of $r$ and therefore is not of weakly coupled type.\par
By the formula \eqref{Ex1F} it is possible to see another typical phenomenon in homogenization of Hamilton-Jacobi
equation, the presence of a flat part in the graph of effective Hamiltonian (see \cite{concordel96}, \cite{lpv86}).

%%%%%%%%%%%%%%%%%%%%%%%%%%%%%%%%%%%%%%%%%%%%%%%%%%%%%%%%%%%%%%%%

\appendix
\section{}\label{Appendix}

\begin{proof}[Proof of Proposition \ref{PrComp}]
We first prove the comparison principle.
Define
\[
 \Psi(x, y,t,s,j)=u_j(x,t)-v_j(y,s)-\frac{|x-y|^2}{2\a}-\frac{|t-s|^2}{2\mu}-\b(|x|^2+|y|^2)-\eta t,
\]
where  $\a$, $\b$, $\mu,$ $\eta$ are positive constants. Since $u,v$ are bounded,
${\rm max}_j\sup_{(\R^N)^2\times [0,T]}\Psi$ is finite and achieved
at some $(\ov x,\ov y,\ov t, \ov s,\ov j).$

For all $j$ and $(x,t)\in\R^N,$ we have
\begin{equation*}
    u_j(x,t)-v_j(x,t) - 2\b|x|^2 -\eta t=\Psi(x,x,t,t,j)\le \Psi(\ov x,\ov y,\ov t,\ov s, \ov
       j)\leq u_{\ov j}(\ov x,\ov t)-v_{\ov j}(\ov y,\ov s).
\end{equation*}
If  $u_{\ov j}(\ov x,\ov t)-v_{\ov j}(\ov y,\ov s)\leq 0$ for all $\b, \eta>0,$
then the comparison holds. Therefore, we suppose that
\begin{eqnarray} \label{CP2}
u_{\ov j}(\ov x,\ov t)-v_{\ov j}(\ov y,\ov s)\geq 0
\end{eqnarray}
for $\b, \eta$ sufficiently small.

The following inequality
\begin{equation*}
    u_1(0,0)-v_1(0,0)  =\Psi(0,0,0,0,1)\le \Psi(\ov x,\ov y,\ov t,\ov s, \ov
       j)
\end{equation*}
and the boundedness of $u,v$ leads to the classical estimates (see \cite[Lemma 4.3]{barles94})
\begin{eqnarray}
&& \b(|\ov x|^2+|\ov y|^2), \frac{|\ov t-\ov s|^2}{2\mu}, \frac{|\ov x-\ov y|^2}{2\a}
\leq 2(|u|_\infty +|v|_\infty), \label{lab-borne} \\
&& \label{lab-lim}
\mathop{\rm lim}_{\mu \to 0} \frac{|\ov t-\ov s|^2}{2\mu}=0, \ \ \
\mathop{\rm lim}_{\b \to 0} \b(|\ov x|+|\ov y|)=0 \ \ \ {\rm and} \ \ \
\mathop{\rm lim}_{\a \to 0}\mathop{\rm lim\,sup}_{\mu, \b \to 0} \frac{|\ov x-\ov y|^2}{2\a} =0
\end{eqnarray}
we will need later.

Assume for a while that it is possible to extract some subsequences $\alpha,\beta, \mu\to 0$
such that
\begin{eqnarray}\label{abs34}
\ov t >0 \ {\rm and} \ \ov s >0.
\end{eqnarray}
It follows that we can write the viscosity inequalities for the
subsolution $u$ and the supersolution $v.$ Setting
$\ov p= {(\ov x -\ov y)}/{\a},$ we have
\begin{equation}\label{CP10}
\frac{ (\ov t -\ov s )}{\a}  +H_{\ov j}\left(\ov x ,\ov t, u(\ov x , \ov t),\ov p+2 \b\ov x\right)\le 0
\end{equation}
and
\begin{equation}\label{CP11}
\frac{ (\ov t  -\ov s )}{\a} +H_{\ov j}\left(\ov y,\ov s,v(\ov y,\ov s),\ov p-2\b \ov y\right)\ge 0.
\end{equation}
Subtracting \eqref{CP11} from \eqref{CP10}, we obtain
\begin{eqnarray} \label{fin111}
\hspace*{0.8cm} \eta \leq H_{\ov j}\left(\ov y,\ov s,v(\ov y,\ov s),\ov p-2\b \ov y\right)
-H_{\ov j}\left(\ov x ,\ov t, u(\ov x , \ov t),\ov p+2 \b\ov x\right)
= \mathcal{T}_1 +\mathcal{T}_2 +\mathcal{T}_3 +\mathcal{T}_4,
\end{eqnarray}
where
\begin{eqnarray*}
&& \mathcal{T}_1 =
 H_{\ov j}\left(\ov y,\ov s, v(\ov y,\ov s),\ov p-2\b \ov y\right)
-H_{\ov j}\left(\ov x,\ov s, v(\ov y,\ov s),\ov p-2\b \ov y\right), \\
&& \mathcal{T}_2 =
 H_{\ov j}\left(\ov x,\ov s, v(\ov y,\ov s),\ov p-2\b \ov y\right)
-H_{\ov j}\left(\ov x,\ov t, v(\ov y,\ov s),\ov p-2\b \ov y\right), \\
&& \mathcal{T}_3 =
 H_{\ov j}\left(\ov x,\ov t, v(\ov y,\ov s),\ov p-2\b \ov y\right)
-H_{\ov j}\left(\ov x,\ov t, v(\ov y,\ov s),\ov p+2\b \ov x\right), \\
&& \mathcal{T}_4 =
 H_{\ov j}\left(\ov x,\ov t, v(\ov y,\ov s),\ov p+2\b \ov x\right)
-H_{\ov j}\left(\ov x,\ov t, u(\ov x,\ov t),\ov p+2\b \ov x\right).
\end{eqnarray*}

From \eqref{lab-borne}, choosing $0<\a, \b <1$ and setting $M=\sqrt{2(|u|_\infty + |v|_\infty)}$
we have
\begin{eqnarray}\label{mie21}
\b |\ov x|, \b |\ov y|\leq M\sqrt{\beta}\leq M
\ \ \ {\rm and}  \ \ \ |\ov p|\leq \frac{M}{\sqrt{\a}}.
\end{eqnarray}

From \eqref{H1}iv), we have
\begin{eqnarray}
\mathcal{T}_1&\leq &\omega ((1+|\ov p|+ 2\beta|\ov y|)|\ov y-\ov x|)\nonumber \\
&\leq & \omega \big( (1+2M)M\sqrt{\a}+\frac{|\ov y-\ov x|^2}{\a}\big)\label{T111}
\end{eqnarray}

If $\a, \b$ are fixed, $\ov x, v(\ov y,\ov s), p-2\b \ov y$ are
bounded independently of $\mu$ by \eqref{lab-borne}. It follows that there
exists a modulus of continuity $\omega_{\a,\b,|v|_\infty,T}$ such that
\begin{eqnarray*}
\mathcal{T}_2\leq \omega_{\a,\b,|v|_\infty,,T} (|\ov s-\ov t|).
\end{eqnarray*}

By \eqref{mie21},
\begin{eqnarray*}
|\ov p-2\b \ov y|, |\ov p+2\b \ov x| \leq  \frac{M}{\sqrt{\a}}+2M
\end{eqnarray*}
and therefore, by \eqref{H1}iii), there exists a modulus of continuity
$\omega_{\a, |v|_\infty,T}$ such that
\begin{eqnarray*}
\mathcal{T}_3 \leq \omega_{\a, |v|_\infty,T}(\b (|\ov x|+|\ov y|))
\leq \omega_{\a, |v|_\infty,T}(2M\b ).
\end{eqnarray*}

The non classical term here is $\mathcal{T}_4$ for which we have to
use \eqref{A3} to deal with:
since $0\leq u_{\ov j}(\ov x,\ov t)-v_{\ov j}(\ov y,\ov s)= \mathop{\rm max}_{1\leq i\leq M}
\{u_i(\ov x,\ov t)-v_i(\ov y,\ov s)\}$ by definition of $\ov j$ and \eqref{CP2}, we obtain
\begin{eqnarray*}
 \mathcal{T}_4 \leq  0.
\end{eqnarray*}

Finally, \eqref{fin111} reads
\begin{eqnarray}\label{fin54}
\hspace*{1.5cm}\eta \leq \omega  \big( (1+2M)M\sqrt{\a}+\frac{|\ov y-\ov x|^2}{\a}\big)
+ \omega_{\a,\b,|v|_\infty,T} (|\ov s-\ov t|)
+ \omega_{\a, |v|_\infty,T}(\b (|\ov x|+|\ov y|)).
\end{eqnarray}
By \eqref{lab-lim}, we can take $\a$ small enough to have
$\omega(\cdots)\leq \eta/3.$ Then we choose successively
$\b$ and $\mu$ small enough to obtain a contradiction
in \eqref{fin54}.

Therefore, choosing
$\mu <\! < \b <\! < \a <\! < \eta$ small enough,
\eqref{abs34} does not hold and, for all extractions, one has for instance $\ov t=0.$
It follows that, for all $j$ and $x,y \in\R^N,$ $t\in [0,T],$ we have
\begin{eqnarray*}
u_j(x,t)-v_j(y,t)\leq \eta t + \frac{|x-y|^2}{2\a}+\b (|x|^2+|y|^2)
+ \left(u_{\ov j}(\ov x,0)-v_{\ov j} (\ov y,\ov s)\right)^+ - \frac{|\ov x-\ov y|^2}{2\a}.
\end{eqnarray*}
Sending $\mu\to 0$ and then $\b\to 0,$ we obtain, using \eqref{lab-lim},
\begin{eqnarray}\label{blabla12}
u_j(x,t)-v_j(y,t)
\leq
\eta t + \frac{|x-y|^2}{2\a}
+ \left(u_{\ov j}(\ov x,0)-v_{\ov j} (\ov y,0)\right)^+ - \frac{|\ov x-\ov y|^2}{2\a}.
\end{eqnarray}
But $u_{\ov j}(\ov x,0)-v_{\ov j} (\ov y,0)\leq u_{0, \ov j}(\ov x)-u_{0, \ov j}(\ov y)$
and by uniform continuity of the $u_{0,j}$'s, $j=1,\cdots,M,$ for all $\rho >0,$
there exists $C_{j,\rho}>0$ such that
\begin{eqnarray}\label{unif-cont11}
u_{0,j}(x)-u_{0,j}(y)\leq \rho + C_{j,\rho} |x-y|
\end{eqnarray}
and therefore
\begin{eqnarray}\label{unif-cont22}
u_{0,j}(x)-u_{0,j}(y)- \frac{|x-y|^2}{2\a}\leq \rho + C_{j,\rho} |x-y| -\frac{|x-y|^2}{2\a}
\leq \rho + \frac{1}{2}\a C_{j,\rho}^2
\end{eqnarray}
We fix $\rho >0$ and set $C_\rho = \mathop{\rm max}_{1\leq i\leq M} C_{i,\rho}/\sqrt{2}.$
Then \eqref{blabla12} becomes
\begin{eqnarray}\label{buc23}
u_j(x,t)-v_j(y,t)
\leq
\rho+\eta t +\a C_{\rho}^2 + \frac{|x-y|^2}{2\a}.
\end{eqnarray}

Using \eqref{lab-lim} and sending succesively $\a\to 0,$ $\eta \to 0,$ $\rho\to 0,$
we conlude that the comparison holds.

By classical Perron's method (see \cite{ik91b}), comparison implies the
existence of a continuous viscosity solution $u$ to \eqref{HJSa}.
Applying the comparison principle again, we obtain the uniqueness
of the solution.
\end{proof}

%%%%%%%%%%%%%%%%%%%%%%%%%%%%%%%%
\begin{proof} [Proof of Proposition \ref{regBUC}]
We first prove that $u$ is bounded. Let
\begin{eqnarray*}
u^\pm (x,t)=(\pm |u_0|_\infty \pm Ct,\cdots , \pm |u_0|_\infty \pm Ct),
\end{eqnarray*}
where $C= C(H, |u_0|_\infty,T)$ is defined by
\begin{eqnarray} \label{def1C}
C:=\sup\left\{ \left|H_j\left(x,t,r,0\right)\right|: \, x\in\R^N,\,t\in [0,T],\,
|r|\le | u_0|_\infty,\,  \,1\leq j\leq M \right\}.
\end{eqnarray}
It suffices to prove that $u^+$ is a supersolution and $u^-$ a subsolution
of \eqref{HJSa}. Then, by the comparison principle of Proposition \ref{PrComp},
we get
\begin{eqnarray*}
u^-\leq u \leq u^+ \quad \R^N\times [0,T]
\end{eqnarray*}
and we obtain the global $L^\infty$ bound
\begin{eqnarray}\label{linfini}
|u|\leq |u_0|_\infty + C(H, |u_0|_\infty,T) T.
\end{eqnarray}
We only prove that $u^+$ is a supersolution, the proof for $u^-$
being similar. At first, $u^+$ satisfies clearly the initial condition.
Since $u^+$ is smooth, for all $j$ and $(x,t)\in \R^N\times (0,T),$
\begin{eqnarray*}
\frac{\partial u_j^+}{\partial t}
+H_j (x,t,u^+(x,t),Du_j^+(x,t))=
C+ H_j (x,t,u^+(x,t),0).
\end{eqnarray*}
But
\begin{eqnarray*}
\mathop{\rm max}_{1\leq k\leq M}\{u_k^+(x,t)-|u_0|_\infty \} = Ct\geq 0
\end{eqnarray*}
is achieved for every index $1\leq k\leq M.$
Therefore, from \eqref{A3}, for all $j,$
\begin{eqnarray*}
 H_j (x,t,u^+(x,t),0)\geq  H_j (x,t,(|u_0|_\infty, \cdots , |u_0|_\infty),0)\geq -C
\end{eqnarray*}
which proves the result.

We prove the uniform continuity of $u$ in the space variable
uniformly in time.
Repeating the proof of the comparison principle with $v=u,$
from \eqref{buc23}, we obtain for all $\rho, \eta >0,$
there exists $C_\rho$ such that, for all $j$, $x,y\in\R^N,$
and $t\in [0,T],$
\begin{eqnarray*}
u_j(x,t)-u_j(y,t)\leq \rho+\eta t +
\mathop{\rm inf}_{\alpha >0}\{ \a C_{\rho}^2 + \frac{|x-y|^2}{2\a}\}
\leq 2\rho+ \sqrt{2} C_\rho |x-y|,
\end{eqnarray*}
if we take $\eta$ such that $\eta T\leq \rho.$
This proves that there exists a modulus of continuity $\omega_{\rm sp}$
in space for $u$ which is independent of $t\in [0,T]:$
\begin{eqnarray*}
u_j(x,t)-u_j(y,t)\leq \omega_{\rm sp}(|x-y|)  \quad x,y\in\R^N,\ t\in [0,T].
\end{eqnarray*}

We continue by deducing a modulus of continuity in time (uniformly
in space). This result is classical in parabolic pdes. Here we adapt
the proof of \cite[Lemma 9.1]{bbl02}.
We want to prove that, for all $\rho >0,$ there exist positive constants
$C_\rho$ and $K_\rho$ such that, for all $j,$
$x,x_0\in\R^N$ with $|x-x_0|\leq 1$ and $0\leq t_0\leq t\leq T,$
\begin{eqnarray}\label{2nd-in}
u_j(x,t)-u_j(x_0,t_0)
\leq \rho+C_\rho|x-x_0|^2+K_\rho(t-t_0).
\end{eqnarray}
and
\begin{eqnarray}\label{3nd-in}
-\rho-C_\rho|x-x_0|^2-K_\rho(t-t_0)\leq u_j(x,t)-u_j(x_0,t_0).
\end{eqnarray}
We will prove only the first inequality, the proof of the second one being
analogous.
Since $x\in \ov B(x_0,1),$
taking
\begin{eqnarray*}
C_\rho\geq {2|u|_\infty},
\end{eqnarray*}
we are sure that \eqref{2nd-in} holds on $\partial B(x_0,1)\times [t_0,T]$
for every $\rho, K_\rho >0.$ It is worth noticing that $C_\rho$ depends only on
$|u|_\infty.$ Next we would like to ensure that \eqref{2nd-in} holds
in $\ov B(x_0,1)\times \{t_0\}.$ To this end, we argue by contradiction
assuming there exists $\rho >0$ such that, for every $C_\rho>0,$ there
exists $j$ and $y_{C_\rho}\in \ov B(x_0,1)$ with
\begin{eqnarray}\label{fg45}
u_j(y_{C_\rho},t_0)-u_j(x_0,t_0)>\rho +C_\rho |y_{C_\rho}-x_0|^2.
\end{eqnarray}
It follows
\begin{eqnarray*}
 |y_{C_\rho}-x_0|\leq \sqrt{\frac{2|u|_\infty}{C_\rho}} \to 0 \ \ {\rm as} \ \
 C_\rho \to +\infty.
\end{eqnarray*}
From \eqref{fg45}, we get
\begin{eqnarray*}
\omega_{\rm sp}(|y_{C_\rho}-x_0|)\geq u_j(y_{C_\rho},t_0)-u_j(x_0,t_0)>\rho
+C_\rho |y_{C_\rho}-x_0|^2\geq \rho,
\end{eqnarray*}
which leads to a contradiction for $C_\rho$ large enough. Note that
the choice of $C_\rho$ to obtain the contradiction depends only on
$\rho, |u|_\infty$ and $\omega_{\rm sp}.$
Finally we proved that, up to choose $C_\rho = C_\rho(\rho, |u|_\infty, \omega_{\rm sp})$
big enough, \eqref{2nd-in} holds on $(\partial B(x_0,1)\times [t_0,T])\cup (
\ov B(x_0,1)\times \{t_0\}).$

For all $1\leq j\leq M,$ we set
\begin{eqnarray*}
\chi_j (y,t):= u_j (x_0,t_0)+\rho + C_\rho|y-x_0|^2+ K_\rho(t-t_0) \quad
(y,t)\in\R^N\times [0,T]
\end{eqnarray*}
and $\chi = (\chi_1, \cdots , \chi_M).$ Note that $\chi$ is a smooth
function.
We claim that we can choose the constant $K_\rho$ big enough in order that
$\chi$ is a strict supersolution of \eqref{HJSa} in $B(x_0,1)\times (t_0,T).$
Indeed, for all $j,$ and $(y,t)\in B(x_0,1)\times (t_0,T),$
\begin{eqnarray}\label{visc09}
\frac{\partial \chi_j}{\partial t}
+H_j (y,t,\chi(y,t),D\chi_j(y,t))= K_\rho + H_j(y,t,\chi(y,t), 2C_\rho (y-x_0)).
\end{eqnarray}
But
\begin{eqnarray*}
\mathop{\rm max}_{1\leq k\leq M}\{\chi_k(y,t)-u_k(x_0,t_0) \} =
\rho + C_\rho|y-x_0|^2+ K_\rho(t-t_0)\geq 0
\end{eqnarray*}
is achieved for every index $1\leq k\leq M.$
Therefore, from \eqref{A3}, for all $j,$
\begin{eqnarray}\label{visc091}
H_j(y,t,\chi(y,t), 2C_\rho (y-x_0))\geq H_j(y,t, u(x_0,t_0), 2C_\rho (y-x_0)).
\end{eqnarray}
By \eqref{H1}iii),
\begin{eqnarray*}
M_{C_\rho, |u|_\infty} := {\rm inf}\{ H_j(y,t,r,p) : y\in\R^N, t\in [0,T], |r|\leq |u|_\infty, |p|\leq
2C_\rho, 1\leq j\leq M\}
\end{eqnarray*}
is finite. Taking
\begin{eqnarray*}
K_\rho > - M_{C_\rho, |u|_\infty},
\end{eqnarray*}
from \eqref{visc09} and \eqref{visc091}, we obtain, for all $j,$
\begin{eqnarray*}
\frac{\partial \chi_j}{\partial t}
+H_j (y,t,\chi(y,t),D\chi_j(y,t))>0 \quad (y,t)\in B(x_0,1)\times (t_0,T)
\end{eqnarray*}
which proves the claim.

From the very definition of viscosity solution, it follows that,
for all $j,$
$\displaystyle \mathop{\rm max}_{\ov B(x_0,1)\times (t_0,T)} \{u_j-\chi_j\}$
is necessarily achieved on the parabolic boundary of $B(x_0,1)\times (t_0,T)$
and therefore \eqref{2nd-in} holds in $\ov B(x_0,R)\times [t_0,T].$

From  \eqref{2nd-in} and  \eqref{3nd-in}, we obtain that, for all $\rho>0,$
$1\leq j\leq M,$ $x\in\R^N,$
and $t,s\in [0,T],$
\begin{eqnarray*}
|u_j(x,t)-u_j(x,s)|
\leq \rho+K_\rho|t-s|
\end{eqnarray*}
and $K_\rho$ is independent of $x.$ This proves the existence of a modulus of
continuity $\omega_{\rm tm}$ in time which is independent of $x\in\R^N.$
\end{proof}

%%%%%%%%%%%%%%%%%%%%%%%%%%%%%%%%%%%%%%%%%%%%%%%%%%%%%%%%%%%%%%%%%%%%%%%%%%%%%%

\end{document}